\documentclass[conference]{IEEEtran}

\usepackage{graphics,stfloats,amssymb,amsmath,amsfonts,epsfig}
\usepackage{algorithm}
\usepackage{algorithmic}
\usepackage{url}

\newcommand{\R}{\mathbb{R}}
\newcommand{\cS}{{\cal S}}
\newcommand{\tv}{\tilde{v}}

\def\beq{\begin{equation}}
\def\eeq{\end{equation}}
\def\eqnok#1{(\ref{#1})}

\newtheorem{theorem}{Theorem}

\begin{document}
% \title{On the Relationship of GROUSE to the ISVD}
\title{On GROUSE and Incremental SVD}

\author{\IEEEauthorblockN{Laura Balzano}
\IEEEauthorblockA{University of Michigan\\
girasole@umich.edu}
\and
\IEEEauthorblockN{Stephen J. Wright}
\IEEEauthorblockA{University of Wisconsin, Madison \\
swright@cs.wisc.edu}}


%Telephone: (800) 555--1212, Fax: (888) 555--1212}


\maketitle

%%%%%%%%%%%%%%%%%%%%%%%%%%%%%%%%%%%%%%%%%%%%%%%%%%%%%%%%%%%%%%%%%%%%%%%%%%%%%%%%
\begin{abstract}

GROUSE (Grassmannian Rank-One Update Subspace
Estimation)~\cite{Balzano10b} is an incremental algorithm for
identifying a subspace of $\R^n$ from a sequence of vectors in this
subspace, where only a subset of components of each vector is revealed
at each iteration. Recent analysis~\cite{grouseconverge} has shown
that GROUSE converges locally at an expected linear rate, under
certain assumptions. GROUSE has a similar flavor to the incremental
singular value decomposition algorithm~\cite{Bunch78}, which updates
the SVD of a matrix following addition of a single column. In this
paper, we modify the incremental SVD approach to handle missing data,
and demonstrate that this modified approach is equivalent to GROUSE,
for a certain choice of an algorithmic parameter.

%% The GROUSE (Grassmannian Rank-One Update Subspace Estimation)
%% algorithm for low-rank matrix completion~\cite{Balzano10b}
%% implements stochastic gradient in order to minimize a data fit cost
%% function over all subspaces of a given dimension. Because the
%% minimization is over all subspaces (i.e., constrained to the
%% Grassmannian), it is non-convex. However, in experiments GROUSE
%% enjoys repeated success when the true data are actually low-rank or
%% when there is a small amount of noise. It has recently been shown
%% that GROUSE has a linear convergence rate in a local region of the
%% global solution~\cite{grouseconverge}. In our pursuit of a deeper
%% understanding, we wish to relate the GROUSE algorithm with the
%% Singular Value Decomposition, another algorithm which minimizes a
%% non-convex cost (again, a data-fit cost function over all possible
%% subspaces) but has a provable polynomial time guarantees in seeking
%% the global minimum. In this paper, we show that GROUSE is in fact
%% very closely related to the incremental Singular Value
%% Decomposition~\cite{Bunch78}.

\end{abstract}

%%%%%%%%%%%%%%%%%%%%%%%%%%%%%%%%%%%%%%%%%%%%%%%%%%%%%%%%%%%%%%%%%%%%%%%%%%%%%%%%
\section{INTRODUCTION}

Subspace estimation and singular value decomposition have been
important tools in linear algebra and data analysis for several
decades. They are used to understand the principal components of a
signal, to reject noise, and to identify best approximations.

The GROUSE (Grassmannian Rank-One Update Subspace Estimation)
algorithm, described in \cite{Balzano10b}, aims to identify a subspace
of low dimension, given data consisting of a sequence of vectors in
the subspace that are missing many of their components. Missing data
is common in such big-data applications as low-cost sensor networks
(in which data often get lost from corruption or bad communication
links), recommender systems (where we are missing consumers' opinions
on products they have yet to try), and health care (where a patient's
health status is only sparsely sampled in time).  GROUSE was developed
originally in an online setting, to be used with streaming data or
when the principal components of the signal may be
time-varying. Several subspace estimation algorithms in the
past~\cite{comon90} have also been developed for the online case and
have even used stochastic gradient, though GROUSE and the approach
described in~\cite{brand2002incremental} are the first to deal with
missing data.

Recent developments in the closely related field of matrix completion
have shown that low-rank matrices can be reconstructed from limited
information, using tractable optimization
formulations~\cite{CandesRecht09,RechtImprovedMC09}.  Given this
experience, it is not surprising that subspace identification is
possible even when the revealed data is incomplete, under appropriate incoherence 
assumptions and using appropriate algorithms.

GROUSE maintains an $n \times d$ matrix with orthonormal columns that
is updated by a rank-one matrix at each iteration. The update strategy
is redolent of other optimization appoaches such as gradient
projection, stochastic gradient, and quasi-Newton methods. It is
related also to the incremental singular value decomposition
approach of \cite{Bunch78}, in which the SVD of a matrix is updated
inexpensively after addition of a column.  We aim in this note to
explore the relationship between the GROUSE and incremental SVD
approaches. We show that when the incremental SVD approach is modified
in a plausible way (to handle missing data, among other issues), we
obtain an algorithm that is equivalent to GROUSE.

\section{GROUSE} \label{sec:grouse:partial}

The GROUSE algorithm was developed for identifying an unknown subspace
$\cS$ of dimension $d$ in $\R^n$ from a sequence of vectors $v_t \in
\cS$ in which only the components indicated by the set $\Omega_t
\subset \{1,\dots,n\}$ are revealed. Specifically, when $\bar{U}$ is
an (unknown) $n \times d$ matrix whose orthonormal columns span $\cS$,
and $s_t \in \R^d$ is a weight vector, we observe the following
subvector at iteration $t$:
\beq \label{eq:vt}
(v_t)_{\Omega_t} = (\bar{U} s_t)_{\Omega_t}
\eeq 
(We use the subscript $\Omega_t$ on a matrix or vector to indicate restriction
to the rows indicated by $\Omega_t$.)

GROUSE is described as Algorithm~\ref{grouse:partial}. It generates a
sequence of $n \times d$ matrices $U_t$ with orthonormal columns,
updating with a rank-one matrix at each iteration in response to the
newly revealed data $(v_t)_{\Omega_t}$. Note that GROUSE makes use of
a steplength parameter $\eta_t$. It was shown in~\cite{grouseconverge}
that GROUSE exhibits local convergence of the range space of $U_t$ to
the range space of $\bar{U}$, at an expected linear rate, under
certain assumptions including incoherence of the subspace $\cS$ with
the coordinate directions, the number of components in $\Omega_t$, and
the choice of steplength parameter $\eta_t$.

\begin{algorithm}
\caption{GROUSE} \label{grouse:partial}
\begin{algorithmic}
\STATE{Given $U_0$, an $n \times d$ orthonormal matrix, with $0<d<n$;}
\STATE{Set $t:=1$;}
\REPEAT 
\STATE{Take $\Omega_t$ and $(v_t)_{\Omega_t}$ from \eqref{eq:vt};}
%\IF{the eigenvalues of $[U_t]_{\Omega_t}^T[U_t]_{\Omega_t}$ lie 
%in the range $[0.5|\Omega_t|/n, 1.5|\Omega_t|/n]$}
\STATE{Define $w_t := \arg \min_w \|[U_t]_{\Omega_t} w - [v_t]_{\Omega_t} \|_2^2$;}
\STATE{Define $p_t := U_t w_t$; 
$[r_t]_{\Omega_t} := [v_t]_{\Omega_t}-[p_t]_{\Omega_t}$; \\
$[r_t]_{\Omega_t^C} := 0$;
$\sigma_t:= \|r_t\| \, \|p_t\|$;}
\STATE{Choose $\eta_t>0$ and set}
\begin{align}
U_{t+1} := U_t &+  \left(\cos (\sigma_t \eta_t)-1 \right) \frac{p_t}{\|p_t\|} \frac{w_t^T}{\|w_t\|} \nonumber \\ 
&+  \sin(\sigma_t \eta_t) \frac{r_t}{\|r_t\|} \frac{w_t^T}{\|w_t\|} \;. \label{eq:gpupdate}
\end{align}
\STATE{$t:=t+1$;}
\UNTIL{termination}
\end{algorithmic}
\end{algorithm}

\section{Incremental Singular Value Decomposition}\label{sec:isvd}

The incremental SVD algorithm of \cite{Bunch78} computes the SVD of a
matrix by adding one (fully observed) column at a time.  The size of
the matrices of left and right singular vectors $U_t$ and $V_t$ grows
as columns are added, as does the diagonal matrix of singular values
$\Sigma_t$. The approach is shown in Algorithm~\ref{isvd:full}. Note
that when the new vector $v_t$ is already in the range space of $U_t$,
we have $r_t=0$, and the basic approach can be modified to avoid
adding an extra dimension to the $U$, $V$, and $\Sigma$ factors in
this situation. If all vectors $v_t$ lie in a subspace $\cS$
of dimension $d$, the modified method will not need to grow $U_t$
beyond size $n \times d$.%, or grow $V_t$ beyond size $d \times d$.

\begin{algorithm}
\caption{Incremental SVD~\cite{Bunch78} } \label{isvd:full}
\begin{algorithmic}
\STATE{Start with null matrixes $U_0$, $V_0$, $\Sigma_0$;}
\STATE{Set $t:=0$;} 
\REPEAT
\STATE{Given new column vector $v_t$;} \STATE{Define $w_t
  := \arg \min_w \|U_t w - v_t\|_2^2 = U_t^T v_t$;} 
\STATE{Define}
\[
p_t := U_t w_t; \quad
r_t := v_t - p_t;
\] 
\STATE{(Set $r_0 := v_0$ when $t=0$);}
\STATE{Noting that}
\[
\left[ \begin{matrix} U_t \Sigma_t V_t^T & v_t \end{matrix} \right] =
\left[ \begin{matrix} U_t & \frac{r_t}{\|r_t\|} \end{matrix} \right]
\left[ \begin{matrix} \Sigma_t & w_t \\ 0 & \|r_t\| \end{matrix}
  \right] \left[ \begin{matrix} V_t &0 \\ 0 & 1 \end{matrix} \right]^T,
\]
\STATE{compute the SVD of the update matrix:}
\begin{equation}
\left[ \begin{matrix} \Sigma_t & w_t \\ 0 & \|r_t\| \end{matrix}
  \right]  =
\hat{U} \hat{\Sigma} \hat{V}^T,
\label{eq:updatewsigma}
\end{equation}
\STATE{and set}
\begin{align*}
U_{t+1} &:=  \left[ \begin{matrix} U_t & \frac{r_t}{\|r_t\|} \end{matrix} \right] \hat{U}, \quad
\Sigma_{t+1} := \hat{\Sigma}, \\
V_{t+1} &:=  \left[ \begin{matrix} V_t &0 \\ 0 & 1 \end{matrix} \right] \hat{V}.
\end{align*}
\STATE{$t:=t+1$;}
\UNTIL{termination}
\end{algorithmic}
\end{algorithm}

\section{Relating  GROUSE to Incremental SVD} \label{sec:isvdgrouse}

Algorithms~\ref{grouse:partial} and \ref{isvd:full} are motivated in
different ways and therefore differ in significant respects. We now
describe a variant --- Algorithm~\ref{isvd:partial} --- that is suited
to the setting addressed by GROUSE, and show that it is in fact
equivalent to GROUSE.  Algorithm~\ref{isvd:partial}, includes the
following modifications.
\begin{itemize}
\item Since only the subvector $(v_t)_{\Omega_t}$ is available, the
  missing components of $v_t$ (corresponding to indices in the
  complement $\Omega_t^C := \{1,2,\dotsc,n\} \setminus \Omega_t$) must
  be ``imputed'' from the revealed components and from the current
  subspace estimate $U_t$.
\item The singular value matrix $\Sigma_t$ is not carried over from
  one iteration to the next. In effect, the singular value estimates
  are all reset to $1$ at each iteration.
\item We allow an arbitrary rotation operator $W_t$ to be applied to
  the columns of $U_t$ at each iteration. This does not affect the
  range space of $U_t$, which is the current estimate of the
  underlying subspace $\cS$.
\item The matrix $U_t$ is not permitted to grow beyond $d$
  columns. 
  %Specifically, we initialize $U_0$ to be an $n \times d$
  %matrix with orthonormal columns, and remove the last column after
  %each update step.
\end{itemize}

%% The proposed algorithm does as follows, essentially merging GROUSE
%% (Algorithm~\ref{grouse:partial}) and iSVD
%% (Algorithm~\ref{isvd:full}). We compute the weights of the projection
%% $w_t$, the in-subspace component $p_t$, and the residual $r_t$ all as
%% in GROUSE, and we use the update of Equation~\eqref{eq:updatewsigma}
%% with one difference: We replace the singular values with an identity
%% matrix. The resulting algorithm is described in
%% Algorithm~\ref{isvd:partial}.

\begin{algorithm}
\caption{iSVD for Partially Observed Vectors} \label{isvd:partial}
\begin{algorithmic}
\STATE{Given $U_0$, an $n \times d$ orthonormal matrix, with $0<d<n$;}
\STATE{Set $t:=1$;}
\REPEAT 
\STATE{Take $\Omega_t$ and $(v_t)_{\Omega_t}$ from \eqref{eq:vt};}
\STATE{Define $w_t := \arg \min_w \|(U_t)_{\Omega_t} w - (v_t)_{\Omega_t} \|_2^2$;}
\STATE{Define}
\begin{align*}
[\tv_t]_i &:= \left\{ \begin{matrix} [v_t]_i & i \in \Omega_t \\ [U_t w_t]_i  & i \in \Omega_t^C \end{matrix} \right. ; \\
p_t &:= U_t w_t; \quad
r_t := \tv_t - p_t;
\end{align*}
% \STATE{Normalize: $v_t \leftarrow v_t/\|v_t\|$;}
\STATE{Noting that}
\[
\left[ \begin{matrix} U_t & \tilde{v}_t \end{matrix} \right] =
\left[ \begin{matrix} U_t & \frac{r_t}{\|r_t\|} \end{matrix} \right] 
\left[ \begin{matrix} I & w_t \\ 0 & \|r_t \| \end{matrix} \right],
\]
\STATE{we compute the SVD of the update matrix:}
\beq \label{eq:isvd.util}
\left[ \begin{matrix} I & w_t \\ 0 & \|r_t \| \end{matrix} \right] =
\tilde{U}_t \tilde{\Sigma}_t \tilde{V}_t^T, 
\eeq
and define $\hat{U}_t$ to be the $(d+1) \times d$ matrix obtained by 
removing the last column from $\tilde{U}_t$.
\STATE{Set $U_{t+1} := \left[ \begin{matrix} U_t &
      \frac{r_t}{\|r_t\|} \end{matrix} \right] \hat{U} W_t$,
  where $W_t$ is an arbitrary $d \times d$ orthogonal matrix.}
\STATE{$t:=t+1$;}
\UNTIL{termination}
\end{algorithmic}
\end{algorithm}

Algorithm~\ref{isvd:partial} is quite similar to an algorithm
proposed in~\cite{brand2002incremental} (see Algorithm~\ref{isvd:brand})
but differs in its handling of the singular
values. In~\cite{brand2002incremental}, the singular values are
carried over from one iteration to the next, but previous estimates
are ``down-weighted'' to place more importance on the vectors
$(v_t)_{\Omega_t}$ from recent iterations. This feature is useful in a scenario in which 
the underlying subspace $\cS$ is changing in time. GROUSE also is influenced more 
by more recent vectors than older ones, thus has a similar (though less explicit) down-weighting feature.

%% The consequence is to downplay the involvement of the initial
%% incomplete vectors on the final subspace
%% estimate. Algorithm~\ref{isvd:partial} also aims to downplay the
%% involvement of initial incomplete vectors via the singular values,
%% but instead of down-weighting the singular values, it simply
%% ignores them by replacing them with the identity matrix.

We show now that for a particular choice of $\eta_t$ in
Algorithm~\ref{isvd:partial}, the Algorithms~\ref{grouse:partial} and
\ref{isvd:partial} are equivalent. Any difference in the updated
estimate $U_{t+1}$ is eliminated when we define the column rotation
matrix $W_t$ appropriately.

\begin{theorem} \label{th:isvd}
Suppose that at iteration $t$ of Algorithms~\ref{grouse:partial} and
\ref{isvd:partial}, the iterates $U_t$ are the same, and the new
observations $v_t$ and $\Omega_t$ are the same. Assume too that $w_t
\neq 0$ and $r_t \neq 0$. Define the following (related) scalar
quantities:
\begin{subequations}
\label{eq:isvd1-4}
\begin{align}
\lambda :=  & \frac12 (\|w_t\|^2 + \|r_t\|^2 + 1) + \nonumber \\
&\frac12 \sqrt{(\|w_t\|^2 + \|r_t\|^2 +1)^2 - 4 \|r_t \|^2}; \label{eq:isvd.1}\\
\label{eq:isvd.2}
\beta := & \frac{\|r_t\|^2+\|w_t\|^2}{\|r_t\|^2+\|w_t\|^2 + (\lambda - \|r_t\|^2)^2} \\
\label{eq:isvd.3}
\alpha := & \frac{\|r_t\| (\lambda - \|r_t\|^2)}{\|r_t\|^2 + \|w_t\|^2+ (\lambda-\|r_t\|^2)^2} \\
\label{eq:isvd.4}
\eta_t := & \frac{1}{\sigma_t} \arcsin \beta = \frac{1}{\sigma_t} \arccos (\alpha \|w_t \|),
\end{align}
\end{subequations}
and define the $d \times d$ orthogonal matrix $W_t$ by
\beq \label{eq:def.Wt}
W_t := \left[ \frac{w_t}{\|w_t\|}  \, | \, Z_t \right],
\eeq
where $Z_t$ is a $d \times d-1$ orthonormal matrix whose columns span
the orthogonal complement of $w_t$. For these choices of $\eta_t$
and $W_t$, the  iterates $U_{t+1}$ generated by
Algorithms~\ref{grouse:partial} and \ref{isvd:partial} are identical.
\end{theorem}
\begin{proof}
We drop the subscript $t$ freely throughout the proof. 

We first derive the structure of the matrix $\hat{U}_t$ in
Algorithm~\ref{isvd:partial}, which is key to the update formula in
this algorithm. We have from \eqnok{eq:isvd.util} that 
\beq \label{eq:isvd.9}
\left[ \begin{matrix} I & w \\ 0 & \|r\| \end{matrix} \right]
\left[ \begin{matrix} I & 0 \\ w^T & \|r\| \end{matrix} \right] =
\left[ \begin{matrix} I + ww^T & \|r\|w \\ \|r\|w^T & \|r\|^2 \end{matrix}\right] =
\tilde{U} \tilde{\Sigma}^2 \tilde{U}^T,
\eeq
and thus the columns of $\tilde{U}$ are eigenvectors of this product
matrix. We see that the columns of the $d \times (d-1)$ orthonormal
matrix $Z_t$ defined in \eqnok{eq:def.Wt} can be used to construct a
set of eigenvectors that correspond to the eigenvalue $1$, since
\beq \label{eq:isvd.10}
\left[ \begin{matrix} I + ww^T & \|r\|w \\ \|r\|w^T & \|r\|^2 \end{matrix}\right]
\left[ \begin{matrix} Z_t \\ 0 \end{matrix} \right] =
\left[ \begin{matrix} Z_t \\ 0 \end{matrix} \right].
\eeq
Two eigenvectors and eigenvalues remain to be determined. Using
$\lambda$ to generally denote one of these two eigenvalues and $(y^T
\, : \, \beta)^T$ to denote the corresponding eigenvector, we have
\beq \label{eq:isvd.11}
\left[ \begin{matrix} I + ww^T & \|r\|w \\ \|r\|w^T & \|r\|^2 \end{matrix}\right]
\left[ \begin{matrix} y \\ \beta \end{matrix} \right] =
\lambda \left[ \begin{matrix} y \\ \beta \end{matrix} \right].
\eeq
The first block row of this expression yields 
\[
y + w (w^Ty + \|r\| \beta) = \lambda y,
\]
which implies that $y$ has the form $\alpha w$ for some $\alpha \in
\R$. By substituting this form into the two block rows from
\eqnok{eq:isvd.11}, we obtain
\begin{align} 
&\alpha (1-\lambda) w + w (\alpha \|w\|^2 + \|r\| \beta) = 0 \;\; \nonumber \\
&\Rightarrow \;\; \alpha (1+\|w\|^2 - \lambda) + \|r\| \beta = 0, \label{eq:isvd.12}
\end{align}
and
\beq \label{eq:isvd.13}
\alpha \| r\| \|w\|^2 + (\|r\|^2 - \lambda) \beta = 0.
\eeq
We require also that  the vector 
\[
\left[ \begin{matrix} y \\ \beta \end{matrix} \right] = 
\left[ \begin{matrix} \alpha w \\ \beta \end{matrix} \right]
\]
has unit norm, yielding the additional condition
\beq \label{eq:isvd.14}
\alpha^2 \|w\|^2 + \beta^2 = 1.
\eeq
(This condition verifies the equality between the ``$\arcsin$'' and
``$\arccos$'' definitions in \eqnok{eq:isvd.4}.)

To find the two possible values for $\lambda$, we seek non-unit roots
of the characteristic polynomial for \eqnok{eq:isvd.9} and make use of
the Schur form
\[
\det \left( \left[ \begin{matrix} A & B \\ C & D \end{matrix} \right] \right) =
(\det D ) \det (A-BD^{-1}C),
\]
to obtain
\begin{align*}
\det & \left[ \begin{matrix} I + ww^T - \lambda I & \|r\| w \\
\|r\|w^T & \|r\|^2 - \lambda \end{matrix} \right] \\
&= (\|r\|^2 - \lambda) \det \left[ (1-\lambda) I + ww^T -
\frac{\|r\|^2}{\|r\|^2-\lambda} ww^T \right] \\
&= (\|r\|^2 - \lambda) \det \left[ (1-\lambda) I -
\frac{\lambda}{\|r\|^2-\lambda} ww^T \right] \\
&= (1-\lambda)^d (\|r\|^2 - \lambda) 
\left( 1- \frac{\lambda \|w\|^2}{(\|r\|^2-\lambda)(1-\lambda)} \right) \\
&= (1-\lambda)^{d-1} \left( (\|r\|^2-\lambda) (1-\lambda) - \lambda \|w\|^2 \right) \\
&= (1-\lambda)^{d-1} (\lambda^2 - \lambda (\|w\|^2 + \|r\|^2+1) + \|r\|^2),
\end{align*}
where we used $\det (I+aa^T) = 1+\|a\|^2$. Thus the two non-unit eigenvalues 
are the roots of the quadratic
\beq \label{eq:isvd.15}
\lambda^2 - \lambda (\|w\|^2 + \|r\|^2+1) + \|r\|^2.
\eeq
When $r \neq 0$ and $w \neq 0$, this quadratic takes on positive
values at $\lambda=0$ and when $\lambda \uparrow \infty$, while the
value at $\lambda=1$ is negative. Hence there are two roots, one in
the interval $(0,1)$ and one in $(1,\infty)$. We fix $\lambda$ to the
larger root, which is given explicitly by \eqnok{eq:isvd.1}. The
corresponding eigenvalue is the first column in the matrix
$\tilde{U}_t$, and thus also in the matrix $\hat{U}_t$. It can be
shown, by reference to formulas \eqnok{eq:isvd.1} and
\eqnok{eq:isvd.15}, that the values of $\beta$ and $\alpha$ defined by
\eqnok{eq:isvd.2} and \eqnok{eq:isvd.3}, respectively, satisfy the
conditions \eqnok{eq:isvd.12}, \eqnok{eq:isvd.13},
\eqnok{eq:isvd.14}. We can now assemble the leading $d$ eigenvectors
of the matrix in \eqnok{eq:isvd.9} to form the matrix $\hat{U}$ as
follows:
\[
\hat{U} := \left[ \begin{matrix} \alpha w & Z_t \\ \beta & 0 \end{matrix}
\right].
\]
Thus, with $W_t$ defined as in \eqnok{eq:def.Wt}, we obtain
\[
\hat{U} W_t^T = 
\left[ \begin{matrix} \alpha w & Z_t \\ \beta & 0 \end{matrix}
\right] 
\left[ \begin{matrix} \frac{w^T}{\|w\|} \\ Z_t^T \end{matrix} \right] =
\left[ \begin{matrix} \frac{\alpha}{\|w\|} ww^T + Z_t Z_t^T \\
\frac{\beta}{\|w\|} w^T \end{matrix} \right].
\]
Therefore, we have from the update formula for
Algorithm~\ref{isvd:partial} that
\begin{align*}
U_{t+1} & = \left[ \begin{matrix} U_t & \frac{r}{\|r\|} \end{matrix} \right]
\hat{U} W_t^T \\
&= U_t \left( \frac{\alpha}{\|w\|} ww^T + Z_t Z_t^T \right) +
\beta \frac{r}{\|r\|} \frac{w^T}{\|w\|}.
\end{align*}
By orthogonality of $W_t$, we have
\[
I = WW^T = \frac{ww^T}{\|w\|^2} + Z_t Z_t^T \; \Rightarrow \;
Z_t Z_t^T - I - \frac{ww^T}{\|w\|^2}.
\]
Hence, by substituting in the expression above, we obtain
\begin{align*}
U_{t+1} &= U_t \left( \alpha \frac{ww^T}{\|w\|} + 
\left( I-\frac{ww^T}{\|w\|^2} \right)  \right) +
\beta \frac{r}{\|r\|} \frac{w^T}{\|w\|} \\
&= U_t + \left[ (\alpha \|w\|-1) \frac{w}{\|w\|} + \beta \frac{r}{\|r\|} \right]
\frac{w^T}{\|w\|},
\end{align*}
which is identical to the update formula in
Algorithm~\ref{grouse:partial} provided that 
\[
\cos \sigma_t \eta_t = \alpha \|w_t \|, \quad
\sin \sigma_t \eta_t = \beta.
\]
These relationships hold because of the definition \eqnok{eq:isvd.4}
and the normality relationship \eqnok{eq:isvd.14}.
\end{proof}

\begin{figure*}
\includegraphics[width=7.1in]{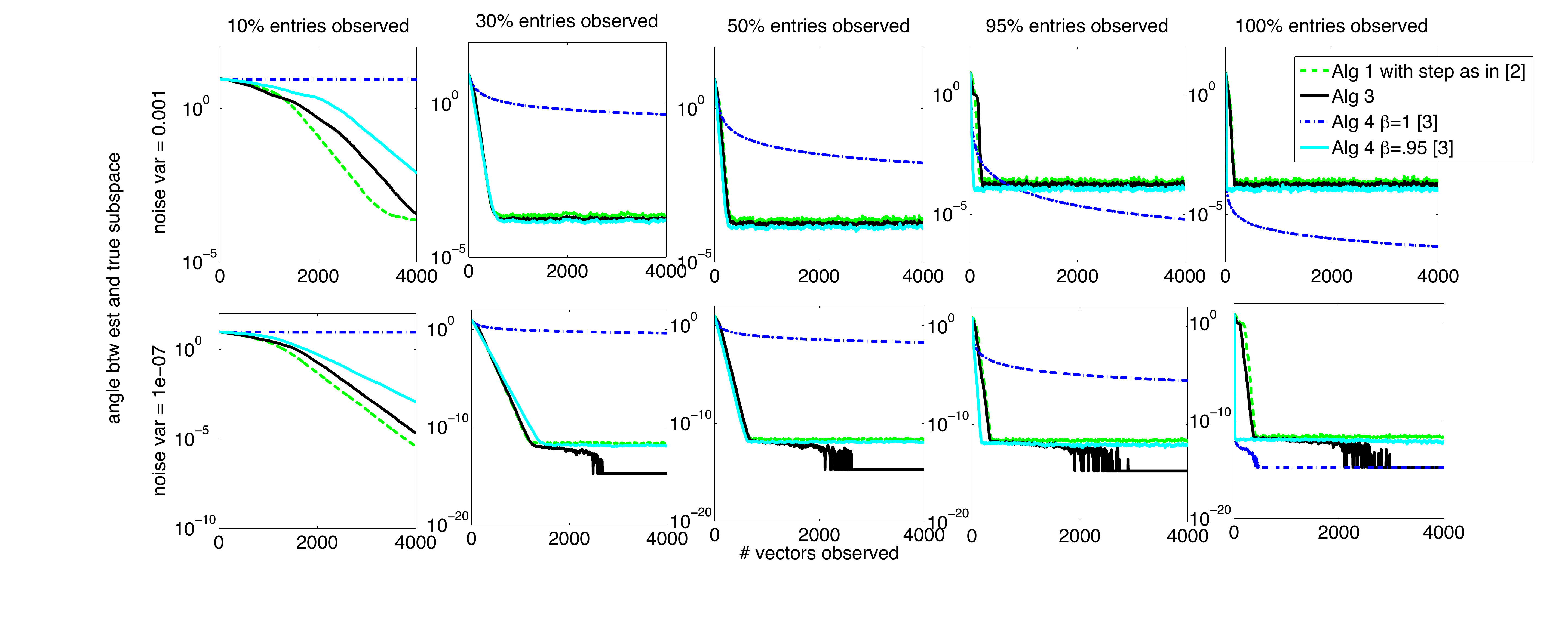} 
\caption{Results for the algorithms described in this
  paper. Algorithm~\ref{isvd:brand} with $\beta=1$ and full data is
  equivalent to the original incermental SVD
  (Algorithm~\ref{isvd:full}). This algorithm performs the best when
  all entries are observed or when just a small amount of data is
  missing and noise is present. Algorithm~\ref{isvd:brand} with
  $\beta=0.95$ and full data at first converges quickly as with
  $\beta=1$ but flatlines much earlier.  GROUSE
  (Algorithm~\ref{grouse:partial}) with the step as prescribed
  in~\cite{grouseconverge} does the best when a very small fraction of
  entries are observed, approaching the theoretical minimum
  (see~\cite{grouseconverge} for details). With low noise and missing
  data, our iSVD method (Algorithm~\ref{isvd:partial}) averages out
  the noise, given enough iterations. Otherwise the algorithms perform
  equivalently. }
\end{figure*}

\begin{algorithm}
\caption{Another iSVD approach for Partial Data~\cite{brand2002incremental} } \label{isvd:brand}
\begin{algorithmic}
\STATE{Given $U_0$, an arbitrary $n \times d$ orthonormal matrix, with $0<d<n$; $\Sigma_0$, a $d \times d$ diagonal matrix of zeros which will later hold the singular values.}
\STATE{Set $t:=1$;}
\REPEAT 
\STATE{Compute $w_t, p_t, r_t$ as in Algorithm~\ref{isvd:partial}.}
\STATE{Compute the SVD of the update matrix:}
\begin{equation*}
\left[ \begin{matrix} \beta \Sigma_t & w_t \\ 0 & \|r_t\| \end{matrix}
  \right]  =
\hat{U} \hat{\Sigma} \hat{V}^T,
\end{equation*}
\STATE{for some scalar $\beta\leq 1$ and set}
\begin{align*}
U_{t+1} &:=  \left[ \begin{matrix} U_t & \frac{r_t}{\|r_t\|} \end{matrix} \right] \hat{U}, \quad
\Sigma_{t+1} := \hat{\Sigma}.
\end{align*}
\STATE{$t:=t+1$;}
\UNTIL{termination}
\end{algorithmic}
\end{algorithm}

\section{Simulations}

To compare the algorithms presented in this note, we ran simulations
as follows. We set $n=200$ and $d=10$, and defined $\bar{U}$ (whose
columns span the target subspce $\cS$) to be a random matrix with
orthonormal columns. The vectors $v_t$ were generated as $\bar{U}
s_t$, where the components of $s_t$ are $\mathcal{N}(0,1)$ i.i.d.  We
also computed a different $n \times d$ matrix with
orthonormal columns, and used that to initialize all algorithms.  
We compared the GROUSE algorithm
(Algorithm~\ref{grouse:partial}) with our proposed missing data iSVD
(Algorithm~\ref{isvd:partial}). Although, as we show in this note,
these algorithms are equivalent for a particular choice of $\eta_t$,
we used the different choice of this parameter prescribed in
~\cite{grouseconverge}.
% \sjwcomment{Referees might ask: Why use the alternative $\eta_t$?
%   Just for variety? Or does it give better performance?}
Finally, we compared to the incomplete data iSVD proposed
in~\cite{brand2002incremental}, which is summarized in
Algorithm~\ref{isvd:brand}. This approach requires a parameter $\beta$
which down-weights old singular value estimates. We obtained the
performance for $\beta=0.95$; performance of this approach degraded
for values of $\beta$ less than $0.9$. The error metric on the y-axis is
$d-\|U_t^T\bar{U}\|_F^2$; see \cite{grouseconverge} for details of
this quantity.

\section{Conclusion}

We have shown an equivalence between GROUSE and a modified incemental
SVD approach. The equivalence is of interest because the two methods
are motivated and constructed from different perspectives --- GROUSE
from an optimization perspective, and incremental SVD from linear
algebra perspective.

%% The GROUSE algorithm incrementally estimates a linear subspace of
%% fixed dimension when the data vectors may have missing components, and
%% initial theory of convergence has been explored. The ISVD does the
%% same and is a theoretically very well-understood algorithm. In this
%% paper we have demonstrated a connection between the two algorithms,
%% which will allow a further exploration of the relationship between
%% stochastic gradient algorithms for subspace estimation and the
%% singular value decomposition.

{\small \bibliographystyle{plain} \bibliography{grouseisvd} }

%\begin{thebibliography}{1}
%\end{thebibliography}

\end{document}